\documentclass[11pt,a4paper]{amsart}
\usepackage{amsmath,amsthm,amsfonts,graphicx,enumerate,url,calc}

\usepackage[colorlinks=true,citecolor=black,linkcolor=black]{hyperref}
\usepackage[lmargin=25mm,rmargin=25mm,tmargin=25mm,bmargin=25mm]{geometry}
\usepackage[numbers,sort&compress]{natbib}

\renewcommand{\baselinestretch}{1.5}
\newcommand{\doi}[1]{\href{http://dx.doi.org/#1}{\texttt{doi:#1}}}

\newcommand{\urlprefix}{}

\newcommand{\half}{\ensuremath{\protect\tfrac{1}{2}}}
\newcommand{\ceil}[1]{\ensuremath{\protect\lceil#1\rceil}}

\newcommand{\floor}[1]{\ensuremath{\protect\lfloor#1\rfloor}}

\newcommand{\EndProof}[1]{
  \begin{minipage}[b]{\textwidth-1cm}
    #1
  \end{minipage}\hfill\qed
  \vspace*{1ex} \renewcommand{\qed}{} }

\newcommand{\tablabel}[1]{\label{tab:#1}}
\newcommand{\tabref}[1]{Table~\ref{tab:#1}}

\newcommand{\Figure}[4][htb]{
  \begin{figure}[#1]
    \vspace*{1ex}
    \begin{center}#3\end{center} \vspace*{-1ex}
    \caption{\figlabel{#2}#4}
  \end{figure}
}

\newcommand{\thmlabel}[1]{\label{thm:#1}}
\newcommand{\thmref}[1]{Theorem~\ref{thm:#1}}

\newcommand{\lemlabel}[1]{\label{lem:#1}}
\newcommand{\lemref}[1]{Lemma~\ref{lem:#1}}

\newcommand{\conjlabel}[1]{\label{con:#1}}
\newcommand{\conjref}[1]{Conjecture~\ref{con:#1}}
\newcommand{\twoconjref}[2]{Conjectures~\ref{con:#1} and \ref{con:#2}}
\newcommand{\threeconjref}[3]{Conjectures~\ref{con:#1}, \ref{con:#2}
  and \ref{con:#3}} 
\newcommand{\eqnlabel}[1]{\label{eqn:#1}}
\newcommand{\eqnref}[1]{\eqref{eqn:#1}}

\newcommand{\figlabel}[1]{\label{fig:#1}}
\newcommand{\figref}[1]{Figure~\ref{fig:#1}}
\newcommand{\seclabel}[1]{\label{sec:#1}}
\newcommand{\secref}[1]{Section~\ref{sec:#1}}

\newcommand{\proplabel}[1]{\label{prop:#1}}
\newcommand{\propref}[1]{Proposition~\ref{prop:#1}}

\renewcommand{\EndProof}[1]{ \smallskip
  \begin{minipage}[b]{\textwidth-1cm}
    #1
  \end{minipage}\hfill\qed
  \smallskip \renewcommand{\qed}{} }

\theoremstyle{plain} 
\newtheorem{theorem}{Theorem}
\newtheorem{lemma}{Lemma}

\theoremstyle{definition} 
\newtheorem{proposition}{Proposition}
\newtheorem{conjecture}{Conjecture}

\begin{document}

\title[On the maximum number of cliques in a graph embedded in  a surface]{On the maximum number of cliques\\ in a graph embedded in  a surface}

\author{Vida Dujmovi\'c}
\address{\newline School of Computer Science
\newline Carleton University
\newline Ottawa, Canada}
\email{vida@cs.mcgill.ca}

\author{Ga\v{s}per Fijav\v{z}}
\address{\newline Faculty of Computer and Information Science
\newline University of Ljubljana\newline Ljubljana, Slovenia}
\email{gasper.fijavz@fri.uni-lj.si}

\author{Gwena\"el Joret}
\address{\newline  D\'epartement d'Informatique 
\newline Universit\'e Libre de Bruxelles
\newline Brussels, Belgium}
\email{gjoret@ulb.ac.be}

\author{Thom Sulanke}
\address{\newline Department of Physics
\newline Indiana University 
\newline Bloomington, Indiana, U.S.A.}
\email{tsulanke@indiana.edu}

\author{David~R.~Wood}
\address{\newline Department of Mathematics and Statistics
\newline The University of Melbourne
\newline Melbourne, Australia}
\email{woodd@unimelb.edu.au}

\thanks{
This work was supported in part by the Actions de Recherche Concert\'ees (ARC) fund of the Communaut\'e fran\c{c}aise de Belgique. 
Vida Dujmovi\'c is supported by the Natural Sciences and Engineering Research Council of Canada.
Ga\v{s}per Fijav\v{z} is supported in part by the Slovenian Research Agency, Research Program P1-0297. 
Gwena\"el Joret is a Postdoctoral Researcher of the Fonds National de la Recherche Scientifique (F.R.S.--FNRS).
David Wood is supported by a QEII Research Fellowship from the Australian Research Council.}
\date{\today}

\begin{abstract}
This paper studies the following question: Given a surface $\Sigma$ and an integer $n$, what is the maximum number of cliques in an $n$-vertex graph embeddable in $\Sigma$? We characterise the extremal graphs for this question, and prove that the answer is between  $8(n-\omega)+2^{\omega}$ and $8n+\frac{5}{2}\,2^{\omega}+o(2^{\omega})$, where $\omega$ is the maximum integer such that the complete graph $K_\omega$ embeds in $\Sigma$. For the surfaces $\mathbb{S}_0$, $\mathbb{S}_1$, $\mathbb{S}_2$, $\mathbb{N}_1$, $\mathbb{N}_2$, $\mathbb{N}_3$ and $\mathbb{N}_4$ we establish an exact answer. \\[3ex]
\end{abstract}

\maketitle

\section{Introduction}

A \emph{clique} in a graph\footnote{We consider simple, finite,
  undirected graphs $G$ with vertex set $V(G)$ and edge set $E(G)$. A
  $K_3$ subgraph of $G$ is called a \emph{triangle} of $G$. For
  background graph theory see \citep{Diestel00}.} is a set of pairwise
adjacent vertices.  Let $c(G)$ be the number of cliques in a graph
$G$.  For example, every set of vertices in the complete graph $K_n$
is a clique, and $c(K_n)=2^n$.  This paper studies the following
question at the intersection of topological and extremal graph theory:
Given a surface $\Sigma$ and an integer $n$, what is the maximum
number of cliques in an $n$-vertex graph embeddable in $\Sigma$?

For previous bounds on the maximum number of cliques in certain graph
families see
\citep{FOT,Eckhoff-DM99,Zykov49,NSTW-JCTB06,ReedWood-TALG,Wood-GC07}
for example.  For background on graphs embedded in surfaces see
\citep{MoharThom,White84}.  Every surface is homeomorphic to
$\mathbb{S}_g$, the orientable surface with $g$ handles, or to
$\mathbb{N}_h$, the non-orientable surface with $h$ crosscaps.  The
\emph{Euler characteristic} of $\mathbb{S}_g$ is $2-2g$.  The
\emph{Euler characteristic} of $\mathbb{N}_h$ is $2-h$.  The
\emph{orientable genus} of a graph $G$ is the minimum integer $g$ such
that $G$ embeds in $\mathbb{S}_g$.  The \emph{non-orientable genus} of
a graph $G$ is the minimum integer $h$ such that $G$ embeds in
$\mathbb{N}_h$. The orientable genus of $K_n$ ($n\geq 3$) is
$\ceil{\frac{1}{12}(n-3)(n-4)}$, and its non-orientable genus is
$\ceil{\frac{1}{6}(n-3)(n-4)}$, except that the non-orientable genus
of $K_7$ is $3$.

Throughout the paper, fix a surface $\Sigma$ with Euler characteristic
$\chi$.  If $\Sigma=\mathbb{S}_0$ then let $\omega=3$, otherwise let
$\omega$ be the maximum integer such that $K_\omega$ embeds in
$\Sigma$.  Thus $\omega=\floor{ \half(7+\sqrt{49-24\chi} ) }$ except
for $\Sigma=\mathbb{S}_0$ and $\Sigma=\mathbb{N}_2$, in which case
$\omega=3$ and $\omega=6$, respectively.


To avoid trivial exceptions, we implicitly assume that $|V(G)|\geq3$
whenever $\Sigma=\mathbb{S}_0$.

Our first main result is to characterise the $n$-vertex graphs
embeddable in $\Sigma$ with the maximum number of cliques; see
\thmref{Extremal} in \secref{Characterisation}.  Using this result we
determine an exact formula for the maximum number of cliques in an
$n$-vertex graph embeddable in each of the the sphere $\mathbb{S}_0$,
the torus $\mathbb{S}_1$, the double torus $\mathbb{S}_2$, the
projective plane $\mathbb{N}_1$, the Klein bottle $\mathbb{N}_2$, as
well as $\mathbb{N}_3$ and $\mathbb{N}_4$; see \secref{LowGenus}.  Our
third main result estimates the maximum number of cliques in terms of
$\omega$.  We prove that the maximum number of cliques in an
$n$-vertex graph embeddable in $\Sigma$ is between
$8(n-\omega)+2^{\omega}$ and
$8n+\frac{5}{2}\,2^{\omega}+o(2^{\omega})$; see \thmref{NumCliques} in
\secref{Bound}.

\section{Characterisation of Extremal Graphs}
\seclabel{Characterisation}

The upper bounds proved in this paper are of the form: every graph $G$
embeddable in $\Sigma$ satisfies $c(G)\leq8|V(G)|+f(\Sigma)$ for some
function $f$.  Define the \emph{excess} of $G$ to be $c(G)-8|V(G)|$.
Thus the excess of $G$ is at most $Q$ if and only if
$c(G)\leq8|V(G)|+Q$.  \thmref{NumCliques} proves that the maximum
excess of a graph embeddable in $\Sigma$ is finite.

In this section we characterise the graphs embeddable in $\Sigma$ with
maximum excess.  A \emph{triangulation} of $\Sigma$ is an embedding of
a graph in $\Sigma$ in which each facial walk has three vertices and
three edges with no repetitions. (We assume that every face of a graph
embedding is homeomorphic to a disc.)\

\begin{lemma}
  \lemlabel{Triangulation} Every graph $G$ embeddable in $\Sigma$ with
  maximum excess is a triangulation of $\Sigma$.
\end{lemma}

\begin{proof}
  Since adding edges within a face increases the number of cliques,
  the vertices on the boundary of each face of $G$ form a clique.

  Suppose that some face $f$ of $G$ has at least four distinct
  vertices in its boundary.  Let $G'$ be the graph obtained from $G$
  by adding one new vertex adjacent to four distinct vertices of $f$.
  Thus $G'$ is embeddable in $\Sigma$, has $|V(G)|+1$ vertices, and
  has $c(G)+16$ cliques, which contradicts the choice of $G$.  Now
  assume that every face of $G$ has at most three distinct vertices.

  Suppose that some face $f$ of $G$ has repeated vertices.  Thus the
  facial walk of $f$ contains vertices $u,v,w,v$ in this order (where
  $v$ is repeated in $f$).  Let $G'$ be the graph obtained from $G$ by
  adding two new vertices $p$ and $q$, where $p$ is adjacent to
  $\{u,v,w,q\}$, and $q$ is adjacent to $\{u,v,w,p\}$.  So $G'$ is
  embeddable in $\Sigma$ and has $|V(G)|+2$ vertices.  If
  $S\subseteq\{p,q\}$ and $S\neq\emptyset$ and $T\subseteq\{u,v,w\}$,
  then $S\cup T$ is a clique of $G'$ but not of $G$.  It follows that
  $G'$ has $c(G)+24$ cliques, which contradicts the choice of $G$.
  Hence no face of $G$ has repeated vertices, and $G$ is a
  triangulation of $\Sigma$.
\end{proof}

Let $G$ be a triangulation of $\Sigma$.  An edge $vw$ of $G$ is
\emph{reducible} if $vw$ is in exactly two triangles in $G$.  We say
$G$ is \emph{irreducible} if no edge of $G$ is reducible
\citep{BE-IJM89,CDP-CGTA04,Sulanke06,Sulanke-KleinBottle,
  Sulanke-Generating,LawNeg-JCTB97,NakaOta-JGT95, JoretWood-JCTB10,
  Lavrenchenko}.  Note that $K_3$ is a triangulation of
$\mathbb{S}_0$, and by the above definition, $K_3$ is irreducible.  In
fact, it is the only irreducible triangulation of $\mathbb{S}_0$.  We
take this somewhat non-standard approach so that \thmref{Extremal}
below holds for all surfaces.

Let $vw$ be a reducible edge of a triangulation $G$ of $\Sigma$.  
Let $vwx$ and $vwy$ be the two
faces incident to $vw$ in $G$.  As illustrated in
\figref{ContractionSplitting}, let $G/vw$ be the graph obtained from
$G$ by \emph{contracting} $vw$; that is, delete the edges $vw,wy,wx$,
and identify $v$ and $w$ into $v$.  $G/vw$ is a simple graph since $x$
and $y$ are the only common neighbours of $v$ and $w$.  Indeed, $G/vw$
is a triangulation of $\Sigma$. Conversely, we say that $G$ is obtained from $G/vw$ by
\emph{splitting} the path $xvy$ at $v$.  If, in addition, $xy\in
E(G)$, then we say that $G$ is obtained from $G/vw$ by
\emph{splitting} the triangle $xvy$ at $v$.  Note that $xvy$ need not
be a face of $G/vw$. In the case that $xvy$ is a face, splitting $xvy$
is equivalent to adding a new vertex adjacent to each of $x,v,y$.

\Figure{ContractionSplitting}{\includegraphics{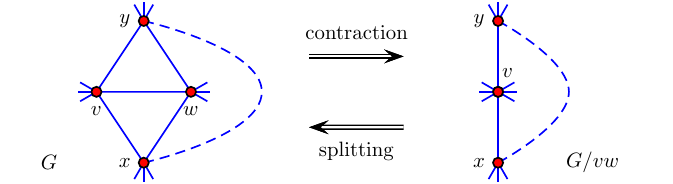}}{Contracting
  a reducible edge.}


Graphs embeddable in $\Sigma$ with maximum excess are characterised in
terms of irreducible triangulations as follows.

\begin{theorem}
  \thmlabel{Extremal} Let $Q$ be the maximum excess of an irreducible
  triangulation of $\Sigma$.  Let $X$ be the set of irreducible
  triangulations of $\Sigma$ with excess $Q$.  Then the excess of
  every graph $G$ embeddable in $\Sigma$ is at most $Q$, with equality
  if and only if $G$ is obtained from some graph in $X$ by repeatedly
  splitting triangles.
\end{theorem}

\begin{proof}
  We proceed by induction on $|V(G)|$.  By \lemref{Triangulation}, we
  may assume that $G$ is a triangulation of $\Sigma$.  If $G$ is
  irreducible, then the claim follows from the definition of $X$ and
  $Q$.  Otherwise, some edge $vw$ of $G$ is in exactly two triangles
  $vwx$ and $vwy$.  By induction, the excess of $G/vw$ is at most $Q$,
  with equality if and only if $G/vw$ is obtained from some $H\in X$
  by repeatedly splitting triangles.  Hence $c(G/vw)\leq8|V(G/vw)|+Q$.

  Observe that every clique of $G$ that is not in $G/vw$ is in
  $\{A\cup\{w\}:A\subseteq\{x,v,y\}\}$.  Thus $c(G)\leq c(G/vw)+8$,
  with equality if and only if $xvy$ is a triangle.  Hence $c(G)\leq
  8|V(G)|+Q$; that is, the excess of $G$ is at most $Q$.

  Now suppose that the excess of $G$ equals $Q$.  Then the excess of
  $G/vw$ equals $Q$, and $c(G)= c(G/vw)+8$ (implying $xvy$ is a
  triangle).  By induction, $G/vw$ is obtained from $H$ by repeatedly
  splitting triangles.  Therefore $G$ is obtained from $H$ by
  repeatedly splitting triangles.

  Conversely, suppose that $G$ is obtained from some $H\in X$ by
  repeatedly splitting triangles.  Then $xvy$ is a triangle and $G/vw$
  is obtained from $H$ by repeatedly splitting triangles.  By
  induction, the excess of $G/vw$ equals $Q$, implying the excess of
  $G$ equals $Q$.
\end{proof}

\section{Low-Genus Surfaces}
\seclabel{LowGenus}

To prove an upper bound on the number of cliques in a graph embedded
in $\Sigma$, by \thmref{Extremal}, it suffices to consider irreducible
triangulations of $\Sigma$ with maximum excess.  The complete list of
irreducible triangulations is known for $\mathbb{S}_0$,
$\mathbb{S}_1$, $\mathbb{S}_2$, $\mathbb{N}_1$, $\mathbb{N}_2$,
$\mathbb{N}_3$ and $\mathbb{N}_4$.  In particular, \citet{SR34} proved
that $K_3$ is the only irreducible triangulation of $\mathbb{S}_0$
(under our definition of irreducible).  \citet{Lavrenchenko} proved
that there are $21$ irreducible triangulations of $\mathbb{S}_1$, each
with between $7$ and $10$ vertices.
\citet{Sulanke-Generating} proved that there are $396,\!784$
irreducible triangulations of $\mathbb{S}_2$, each with between $10$
and $17$ vertices.  \citet{Barnette-JCTB82} proved that the embeddings
of $K_6$ and $K_7-K_3$ in $\mathbb{N}_1$ are the only irreducible
triangulations of $\mathbb{N}_1$.  \citet{Sulanke-KleinBottle} proved
that there are $29$ irreducible triangulations of $\mathbb{N}_2$, each
with between $8$ and $11$ vertices (correcting an earlier result by
\citet{LawNeg-JCTB97}).  \citet{Sulanke-Generating} proved that there
are $9708$ irreducible triangulations of $\mathbb{N}_3$, each with
between $9$ and $16$ vertices.  \citet{Sulanke-Generating} proved that
there are $6,\!297,\!982$ irreducible triangulations of
$\mathbb{N}_4$, each with between $9$ and $22$ vertices.  Using the
lists of all irreducible triangulations due to \citet{Sulanke-Web} and
a naive algorithm for counting cliques\footnote{The code is available
  from the authors upon request.}, we have computed the set $X$ in
\thmref{Extremal} for each of the above surfaces; see \tabref{table}.
This data with \thmref{Extremal} implies the following results.

\renewcommand{\tabcolsep}{2.8pt}

\begin{table}[!h]
  \caption{\tablabel{table}The maximum excess of an $n$-vertex irreducible triangulation of $\Sigma$.}
  \begin{tabular}{ccc|cccccccccccccccccc|c}
    \hline
    $\Sigma$         & $\chi$&$\omega$& $n=3$ & $6$  & $7$  & $8$  & $9$  & $10$ & $11$& $12$& $13$& $14$  & $15$& $16$& $17$  &  $18$  &  $19$  &  $20$  &  $21$  &  $22$  & max\\\hline
    $\mathbb{S}_0$&$2$     &$3$          & {\boldmath $-16$} &         &         &         &         &          &         &         &         &               &         &          &           &           &           &         &         &         &     $-16$             \\
    $\mathbb{S}_1$&$0$     &$7$          &            &         & \textbf{72}    &  48   & 40    & 32      &         &         &         &               &         &         &           &           &           &         &         &         &     72             \\
    $\mathbb{S}_2$&$-2$   &$8$           &           &         &         &         &         & \textbf{208}    & 160  & 136  & 128  & 120        & 96    & 88    & 80      &          &           &         &         &         &       208          \\
    $\mathbb{N}_1$&$1$   &$6$           &            & \textbf{16}     &  8     &         &         &          &          &         &         &               &         &         &           &          &           &         &         &         &   16              \\
    $\mathbb{N}_2$&$0$   &$6$           &            &          &        & \textbf{48}    & \textbf{48}     & 40     & 32    &         &          &               &         &         &           &          &           &         &         &         &   48              \\
    $\mathbb{N}_3$&$-1$   &$7$           &          &          &        &         & \textbf{104}     & \textbf{104} & 96    & 80    &   80  &    72       & 64    &  56   &           &          &           &         &         &         &    104             \\
    $\mathbb{N}_4$&$-2$   &$8$           &          &          &        &         & \textbf{216}     & 208 & 152  & 136  &   136 &  136      & 128    &  120&  112 & 107    &  99     & 91   & 83       &  75      &     216            \\
    \hline
  \end{tabular}
\end{table}

\begin{proposition} 
\proplabel{Planar}
  Every planar graph $G$ with $|V(G)|\geq3$ has at most $8|V(G)|-16$
  cliques, as proved by \citet{Wood-GC07}. Moreover, a planar graph
  $G$ has $8|V(G)|-16$ cliques if and only if $G$ is obtained from the
  embedding of $K_3$ in $\mathbb{S}_0$ by repeatedly splitting
  triangles.
\end{proposition}

\begin{proposition} 
\proplabel{Toroidal}
  Every toroidal graph $G$ has at most $8|V(G)|+72$ cliques. Moreover,
  a toroidal graph $G$ has $8|V(G)|+72$ cliques if and only if $G$ is
  obtained from the embedding of $K_7$ in $\mathbb{S}_1$ by repeatedly splitting
  triangles  (see \figref{K6K7}).  
\end{proposition}

\begin{figure}[!h]
  \begin{center}
  \includegraphics{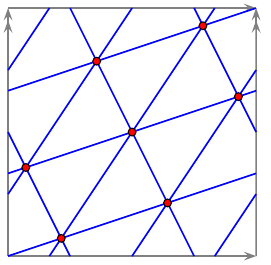}
\quad
    \includegraphics{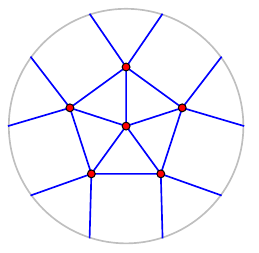}
\hfill
  \end{center}
\vspace{-3ex}
  \caption{\figlabel{K6K7} $K_7$ embedded in the torus, and $K_6$ embedded in the projective plane.}
\end{figure}

\begin{proposition} 
\proplabel{DoubleTorus}
  Every graph $G$ embeddable in $\mathbb{S}_2$ has at most
  $8|V(G)|+208$ cliques. Moreover, a graph $G$ embeddable in
  $\mathbb{S}_2$ has $8|V(G)|+208$ cliques if and only if $G$ is
  obtained from one of the following two graph embeddings in
  $\mathbb{S}_2$ by splitting repeatedly triangles\footnote{This representation
    describes a graph with vertex set $\{a,b,c,\dots\}$ by adjacency
    lists of the vertices in order $a,b,c,\dots$. The graph \# refers
    to the position in Sulanke's file \citep{Sulanke-Web}.}:

  graph \#1:
  bcde,aefdghic,abiehfgd,acgbfihe,adhcigfb,begchjid,bdcfeijh,bgjfcedi,bhdfjgec,fhgi

  graph \#6:
  bcde,aefdghijc,abjehfgd,acgbfjihe,adhcjgfb,begchjd,bdcfejh,bgjfcedi,bhdj,bidfhgec
\end{proposition}

\begin{proposition} 
\proplabel{ProjectivePlanar}
  Every projective planar graph $G$ has at most $8|V(G)|+16$
  cliques. Moreover, a projective planar graph $G$ has $8|V(G)|+16$
  cliques if and only if $G$ is obtained from the embedding of $K_6$
  in $\mathbb{N}_1$ by repeatedly splitting triangles  (see \figref{K6K7}).
\end{proposition}

\begin{proposition} 
\proplabel{Klein}
  Every graph $G$ embeddable in the Klein bottle $\mathbb{N}_2$ has at most
  $8|V(G)|+48$ cliques. Moreover, a graph $G$ embeddable in
  $\mathbb{N}_2$ has $8|V(G)|+48$ cliques if and only $G$ is obtained
  from one of the following three graph embeddings in $\mathbb{N}_2$
  by repeatedly splitting triangles (see \figref{Klein}):

  graph \#3: bcdef,afgdhec,abefd,acfhbge,adghbcf,aecdhgb,bfhed,bdfge

  graph \#6: bcde,aefdghc,abhegd,acgbfhe,adhcgfb,beghd,bdcefh,bgfdec

  graph \#26:
  bcdef,afghidec,abefd,acfhgibe,adbcf,aecdhigb,bfidh,bgdfi,bhfgd
\end{proposition}

\begin{figure}[!h]
  \begin{center}
    \includegraphics{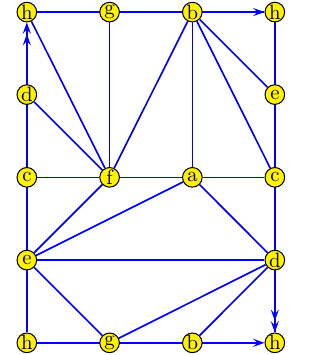}
    \includegraphics{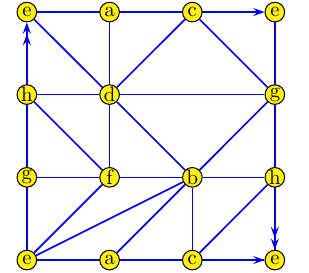}
    \includegraphics{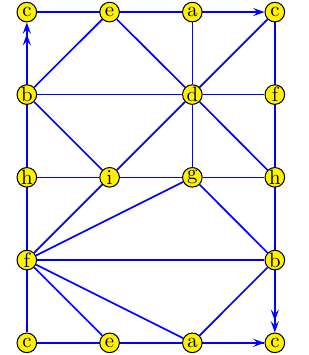}
  \end{center}
 \caption{\figlabel{Klein} Irreducible triangulations of
    $\mathbb{N}_2$ with maximum excess: left-to-right \#3, \#6, \#26.}
\end{figure}

\begin{proposition} 
\proplabel{N3}
  Every graph $G$ embeddable in $\mathbb{N}_3$ has at most
  $8|V(G)|+104$ cliques. Moreover, a graph $G$ embeddable in
  $\mathbb{N}_3$ has $8|V(G)|+104$ cliques if and only $G$ is obtained
  from one of the following 15 graph embeddings in $\mathbb{N}_3$ by
repeatedly   splitting triangles:

  graph \#1:
  bcde,aefdghic,abiegfd,acfbgie,adicghfb,behigcd,bdifceh,bgefi,bhfgdec

  graph \#3:
  bcde,aefdghic,abiehd,achfbgie,adichgfb,begihd,bdifeh,bgecdfi,bhfgdec

  graph \#4:
  bcde,aefdghic,abiehd,achifbge,adgichfb,behgid,bdeifh,bgfecdi,bhdfgec

  graph \#6:
  bcde,aefdghic,abiehfd,acfbgihe,adhcifb,beighcd,bdifh,bgfcedi,bhdgfec

  graph \#8:
  bcde,aefdghic,abiehgfd,acfbgihe,adhcigfb,begcd,bdiefch,bgcedi,bhdgec

  graph \#10:
  bcde,aefdghic,abifegd,acgbfhie,adigcfb,becighd,bdceifh,bgfdi,bhdegfc

  graph \#12:
  bcde,aefdghic,abifehd,achfbgie,adihcfb,becighd,bdifh,bgfdcei,bhedgfc

  graph \#14:
  bcde,aefdghic,abigehd,achfbgie,adihcgfb,beghd,bdicefh,bgfdcei,bhedgc

  graph \#16:
  bcde,aefgdhic,abiegd,acgbhfie,adicghfb,behdig,bfihecd,bdfegi,bhgfdec

  graph \#19:
  bcde,aefghdic,abiehd,achbifge,adgichfb,behidg,bfdeih,bgifecd,bdfhgec

  graph \#20:
  bcde,aefghdic,abigehd,achbifge,adgchifb,beidg,bfdecih,bgiecd,bdfehgc

  graph \#21:
  bcde,aefghdic,abihegd,acgfhbie,adigchfb,behdg,bfdceih,bgicefd,bdeghc

  graph \#22:
  bcde,aefghdic,abihegd,acgfibhe,adhcgifb,beidg,bfdceih,bgiced,bdfeghc

  graph \#82:
  bcdef,afgdhiec,abefd,acfigbhe,adhgibcf,aecdihgb,bfheid,bdegfi,bhfdge

  graph \#2464:
  bcdef,afghijdec,abefd,acfhigjbe,adbcf,aecdhjigb,bfidjh,bgjfdi,bhdgfj,bifhgd
\end{proposition}

\begin{proposition} 
  \proplabel{N4} Every graph $G$ embeddable in $\mathbb{N}_4$ has at
  most $8|V(G)|+216$ cliques. Moreover, a graph $G$ embeddable in
  $\mathbb{N}_4$ has $8|V(G)|+216$ cliques if and only if $G$ is
  obtained from one of the following three graph embeddings in
  $\mathbb{N}_4$ by repeatedly splitting triangles:

  graph \#1:
  bcdef,afdgehic,abiegfhd,achgbfie,adicgbhf,aehcgidb,bdhifce,befcdgi,bhgfdec

  graph \#2:
  bcdef,afdgehic,abifehgd,acgbfhie,adigbhcf,aecighdb,bdchfie,becgfdi,bhdegfc

  graph \#3:
  bcdef,afdgheic,abihfegd,acgbfihe,adhbigcf,aechgidb,bdceifh,bgfcide,begfdhc
\end{proposition}

Note that the three embeddings in \propref{N4} are of the same graph.

\section{A Bound for all Surfaces}
\seclabel{Bound}

Recall that $\Sigma$ is a surface with Euler characteristic $\chi$,
and if $\Sigma=\mathbb{S}_0$ then $\omega=3$, otherwise $\omega$ is
the maximum integer such that $K_\omega$ embeds in $\Sigma$.  We start
with the following upper bound on the minimum degree of a graph.

\begin{lemma}
  \lemlabel{SmallDegree} Assume $\Sigma \neq \mathbb{S}_{0}$. Then
  every graph $G$ embeddable in $\Sigma$ has minimum degree at most
  \begin{equation*}
    6+\frac{\omega^2-5\omega-7}{|V(G)|}\enspace.
  \end{equation*}
\end{lemma}

\begin{proof} 
  By the definition of $\omega$, the complete graph $K_{\omega + 1}$
  cannot be embedded in $\Sigma$.  Thus if $\Sigma=\mathbb{S}_g$ then
  $g = \half(2-\chi) \leq\ceil{\frac{1}{12}(\omega-2)(\omega-3)} - 1$,
  and if $\Sigma=\mathbb{N}_h$ then $h=2-\chi \leq
  \ceil{\frac{1}{6}(\omega-2)(\omega-3)} - 1$.  In each case, it
  follows that $2 - \chi \leq\frac{1}{6}(\omega-2)(\omega-3) -
  \frac{1}{6}$. That is,
  \begin{equation}
    \eqnlabel{OmegaChi}
    -6\chi\leq\omega^2-5\omega-7\enspace.
  \end{equation}

  Say $G$ has minimum degree $d$. It follows from Euler's Formula that
  $|E(G)|\leq3 |V(G)|-3\chi$. By \eqnref{OmegaChi},

\EndProof{
  \begin{equation*}
    \eqnlabel{SmallDegree}
    d\leq \frac{2 |E(G)|}{|V(G)|}\leq\frac{6 |V(G)|-6\chi}{|V(G)|}\leq 6+\frac{\omega^2-5\omega-7}{|V(G)|}\enspace.
  \end{equation*}}
\end{proof}

For graphs in which the number of vertices is slightly more than
$\omega$, \lemref{SmallDegree} can be reinterpreted as follows.

\begin{lemma}
  \lemlabel{SmallDegreeJ} Assume $\Sigma \neq \mathbb{S}_{0}$. Let
  $s:=\ceil{\sqrt{\omega + 11}-3}\geq 1$.  Let $G$ be a graph
  embeddable in $\Sigma$.  If $G$ has at most $\omega+1$ vertices,
  then $G$ has minimum degree at most $\omega - 1$.  If $G$ has at
  least $\omega+j$ vertices, where $j\in[2,s]$, then $G$ has minimum
  degree at most $\omega - j + 1$.
\end{lemma}

\begin{proof}
  Say $G$ has minimum degree $d$. If $|V(G)| \leq \omega$, then
  trivially $d \leq \omega - 1$.  If $|V(G)| = \omega + 1$, then $G$
  is not complete (by the definition of $\omega$), again implying that
  $d \leq \omega - 1$. Now assume $|V(G)| \geq \omega + j$ for some
  $j\in[2,s]$. By \lemref{SmallDegree},
$$
d\leq 6+\frac{\omega^2-5\omega-7}{\omega+j}=\omega-j + 1 +\frac{j^{2}
  + 5j - 7}{\omega+j} \enspace.
$$
Since $j\leq s < \sqrt{\omega + 11} - 2$, we have $j^{2} + 5j - 7 \leq
s^{2} + 4s - 7 + j < \omega + j$.  It follows that $d\leq\omega-j+1$.
\end{proof}

Now we prove our first upper bound on the number of cliques.

\begin{lemma}
  \lemlabel{FirstUpperBound} Assume $\Sigma \neq \mathbb{S}_{0}$. Let
  $s:=\ceil{\sqrt{\omega + 11}-3}\geq 1$.  Let $G$ be an $n$-vertex
  graph embeddable in $\Sigma$. Then
$$
c(G)\leq \left\{
  \begin{array}{ll}
    \frac{5}{2}\,2^{\omega} & \text{ if } n \leq \omega+s,      \\
    \frac{5}{2}\,2^{\omega}+(n-\omega-s)2^{\omega-s+1}  & \text{ otherwise.}    
  \end{array}
\right.
$$
\end{lemma}

\begin{proof}
  Let $v_1,v_2,\dots,v_n$ be an ordering of the vertices of $G$ such
  that $v_i$ has minimum degree in the subgraph
  $G_i:=G-\{v_1,\dots,v_{i-1}\}$.  Let $d_i$ be the degree of $v_i$ in
  $G_i$ (which equals the minimum degree of $G_i$).  Charge each
  non-empty clique $C$ in $G$ to the vertex $v_i\in C$ with $i$
  minimum.  Charge the clique $\emptyset$ to $v_n$.

  We distinguish three types of vertices.  Vertex $v_i$ is type-1 if
  $i\in[1,n-\omega-s]$.  Vertex $v_i$ is type-2 if
  $i\in[n-\omega-s+1,n-\omega]$.  Vertex $v_i$ is type-3 if
  $i\in[n-\omega+1,n]$.

  Each clique charged to a type-3 vertex is contained in
  $\{v_{n-\omega+1},\dots,v_n\}$, and there are at most $2^{\omega}$
  such cliques.

  Say $C$ is a clique charged to a type-1 or type-2 vertex $v_i$.
  Then $C-\{v_i\}$ is contained in $N_{G_i}(v_i)$, which consists of
  $d_i$ vertices.  Thus the number of cliques charged to $v_i$ is at
  most $2^{d_i}$.  Recall that $d_i$ equals the minimum degree of
  $G_i$, which has $n-i+1$ vertices.

  If $v_i$ is type-2 then, by \lemref{SmallDegreeJ} with
  $j=n-\omega-i+1\in[1,s]$, we have $d_i\leq \omega-j + 1$, and
  $d_i\leq \omega-j$ if $j = 1$.  Thus the number of cliques charged
  to type-2 vertices is at most
$$2^{\omega-1} + \sum_{j=2}^s2^{\omega-j+1}\leq 2^{\omega-1} + \sum_{j=1}^{\omega-1}2^j
<\frac{3}{2}\,2^{\omega}\enspace.$$

If $v_i$ is type-1 then $G_i$ has more than $\omega+s$ vertices, and
thus $d_i\leq \omega-s+1$ by \lemref{SmallDegreeJ} with $j=s$.  Thus
the number of cliques charged to type-1 vertices is at most
$(n-\omega-s)2^{\omega-s+1}$.
\end{proof}

We now prove the main result of this section; it provides lower and 
upper bound on the maximum number of cliques in a  graph embeddable in
$\Sigma$.


\begin{theorem}
  \thmlabel{NumCliques} Every $n$-vertex graph embeddable in $\Sigma$
  contains at most $8n+\frac{5}{2}\,2^{\omega}+o(2^{\omega})$ cliques.
  Moreover, for each $n\geq \omega$, there is an $n$-vertex graph
  embeddable in $\Sigma$ with $8(n-\omega)+2^{\omega}$ cliques.
\end{theorem}

\begin{proof}
 To prove the upper bound, we may assume that
  $\Sigma\neq\mathbb{S}_0$, and by \thmref{Extremal}, we need only
  consider $n$-vertex irreducible triangulations of
  $\Sigma$. \citet{JoretWood-JCTB10} proved that, in this case, $n\leq
  22-13\chi$.
  By \eqnref{OmegaChi},
  \begin{equation*}
    \eqnlabel{NakaOta}
    n\leq 22-13\chi\leq 22+\tfrac{13}{6}(\omega^2-5\omega-7)< 3\omega^2\enspace.
  \end{equation*}
  If $n \leq \omega + s$ then $c(G) \leq \tfrac{5}{2}\,2^{\omega}$ by
  \lemref{FirstUpperBound}.  If $n > \omega + s$ then by the same
  lemma,
$$c(G)\leq  
\tfrac{5}{2}\,2^{\omega}+(3\omega^2-\omega-s)2^{\omega-s+1}<
\tfrac{5}{2}\,2^{\omega} +3\omega^2\,2^{\omega-s+1}<
\tfrac{5}{2}\,2^{\omega}+2^{\omega-s+2\log\omega+3}\enspace.$$ Since
$s\in\Theta(\sqrt{\omega})$, we have $c(G)\leq
\frac{5}{2}\,2^{\omega}+o(2^{\omega})$.

  To prove the lower bound, start with $K_\omega$ embedded in $\Sigma$
  (which has $2^\omega$ cliques). Now, while there are less than
  $n$ vertices, insert a new vertex adjacent to each vertex of a
  single face. Each new vertex adds at least 8 new cliques.  Thus we
  obtain an $n$-vertex graph embedded in $\Sigma$ with at least
  $8(n-\omega)+2^{\omega}$ cliques.
\end{proof}

\section{Concluding Conjectures}

We conjecture that the upper bound in \thmref{NumCliques} can be
improved to more closely match the lower bound.

\begin{conjecture}
  \conjlabel{ImprovedUpperBound} Every graph $G$ embeddable in
  $\Sigma$ has at most $8|V(G)|+2^{\omega}+o(2^{\omega})$ cliques.
\end{conjecture}

If $K_\omega$ triangulates $\Sigma$, then we conjecture the following
exact answer.

\begin{conjecture}
  \conjlabel{ExactUpperBound} Suppose that $K_\omega$ triangulates
  $\Sigma$. Then every graph $G$ embeddable in $\Sigma$ has at most
  $8(|V(G)|-\omega)+2^{\omega}$ cliques, with equality if and only if
  $G$ is obtained from $K_{\omega}$ by repeatedly splitting triangles.
\end{conjecture}

By \thmref{Extremal}, this conjecture is equivalent to:

\begin{conjecture}
  \conjlabel{Only} Suppose that $K_\omega$ triangulates $\Sigma$. Then
  $K_\omega$ is the only irreducible triangulation of $\Sigma$ with
  maximum excess.
\end{conjecture}

The results in \secref{LowGenus} confirm
\twoconjref{ExactUpperBound}{Only} for $\mathbb{S}_0$, $\mathbb{S}_1$
and $\mathbb{N}_1$.

Now consider surfaces possibly with no complete graph
triangulation. Then the bound $c(G)\leq 8(|V(G)|-\omega)+2^{\omega}$
(in \conjref{ExactUpperBound}) is false for $\mathbb{S}_2$,
$\mathbb{N}_2$, $\mathbb{N}_3$ and $\mathbb{N}_4$. Loosely speaking,
this is because these surfaces have `small' $\omega$ compared to
$\chi$. In particular, $\omega=\floor{\half(7+\sqrt{49-24\chi})}$
except for $\mathbb{S}_0$ and $\mathbb{N}_2$, and
$\omega=\half(7+\sqrt{49-24\chi})$ if and only if $K_\omega$
triangulates $\Sigma\neq\mathbb{S}_0$.  This phenomenon motivates the
following conjecture:

\begin{conjecture}
  \conjlabel{GeneralExactUpperBound} Every graph $G$ embeddable in
  $\Sigma$ has at
  most $$8|V(G)|-4(7+\sqrt{49-24\chi})+2^{(7+\sqrt{49-24\chi})/2}$$
  cliques, with equality if and only if $K_\omega$ triangulates
  $\Sigma$ and $G$ is obtained from $K_{\omega}$ by repeatedly
  splitting triangles.
\end{conjecture}

There are two irreducible triangulations of $\mathbb{S}_2$ with
maximum excess, there are three irreducible triangulations of
$\mathbb{N}_2$ with maximum excess, there are 15 irreducible
triangulations of $\mathbb{N}_3$ with maximum excess, and there are
three irreducible triangulations of $\mathbb{N}_4$ with maximum
excess. This suggests that for surfaces with no complete graph
triangulation, a succinct characterisation of the extremal examples
(as in \conjref{Only}) might be difficult. Nevertheless, we conjecture
the following strengthening of \conjref{Only} for all surfaces.

\begin{conjecture}
  \conjlabel{OnlyOnly} Every irreducible triangulation of $\Sigma$
  with maximum excess contains $K_\omega$ as a subgraph.
\end{conjecture}

A triangulation of a surface $\Sigma$ is \emph{vertex-minimal} if it
has the minimum number of vertices in a triangulation of $\Sigma$. Of
course, every vertex-minimal triangulation is
irreducible. \citet{Ringel55} and \citet{JR80} together proved that
the order of a vertex-minimal triangulation is $\omega$ if $K_\omega$
triangulates $\Sigma$, is $\omega+2$ if
$\Sigma\in\{\mathbb{S}_2,\mathbb{N}_2,\mathbb{N}_3\}$, and is
$\omega+1$ for every other surface.

Triangulations \#26 of $\mathbb{N}_2$ and \#2464 of $\mathbb{N}_3$ are
the only triangulations in Propositions~\ref{prop:Planar}--\ref{prop:N4} that are not vertex
minimal. Triangulation \#26 of $\mathbb{N}_2$ is obtained from two
embeddings of $K_6$ in $\mathbb{N}_1$ joined at the face $bdf$ (see
\figref{Klein}). Triangulation \#2464 of $\mathbb{N}_3$ is obtained by
joining an embedding of $K_6$ in $\mathbb{N}_1$ and an embedding of
$K_7$ in $\mathbb{S}_1$  at the face $bdf$. 

  \begin{figure}[!h]
    \vspace*{1ex}
    \begin{center}\includegraphics{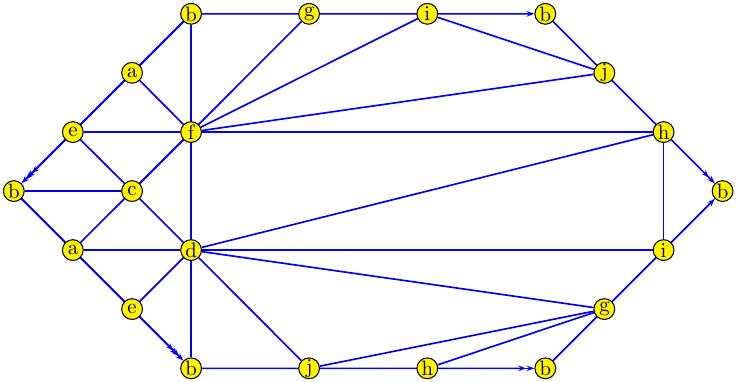}\end{center} 
\vspace*{-2ex}
    \caption{\figlabel{2464} Triangulation \#2464 of $\mathbb{N}_3$.}
  \end{figure}

Every other triangulation in Propositions~\ref{prop:Planar}--\ref{prop:N4} is obtained
from an embedding of $K_\omega$ by adding (at most two) vertices and 
edges until a vertex minimal triangulation is obtained. 
This provides some evidence for our final conjecture:



\begin{conjecture}
  \conjlabel{OnlyOnlyOnly} For every surface $\Sigma$, the maximum
  excess is attained by some vertex-minimal triangulation of $\Sigma$
  that contains $K_\omega$ as a subgraph. Moreover, if
  $\Sigma\not\in\{\mathbb{N}_2,\mathbb{N}_3\}$ then every irreducible
  triangulation with maximum excess is vertex-minimal and contains
  $K_\omega$ as a subgraph.
\end{conjecture}

We have verified
\threeconjref{GeneralExactUpperBound}{OnlyOnly}{OnlyOnlyOnly} for
$\mathbb{S}_0$, $\mathbb{S}_1$, $\mathbb{S}_2$, $\mathbb{N}_1$,
$\mathbb{N}_2$, $\mathbb{N}_3$ and $\mathbb{N}_4$.


\def\cprime{$'$} \def\soft#1{\leavevmode\setbox0=\hbox{h}\dimen7=\ht0\advance
  \dimen7 by-1ex\relax\if t#1\relax\rlap{\raise.6\dimen7
  \hbox{\kern.3ex\char'47}}#1\relax\else\if T#1\relax
  \rlap{\raise.5\dimen7\hbox{\kern1.3ex\char'47}}#1\relax \else\if
  d#1\relax\rlap{\raise.5\dimen7\hbox{\kern.9ex \char'47}}#1\relax\else\if
  D#1\relax\rlap{\raise.5\dimen7 \hbox{\kern1.4ex\char'47}}#1\relax\else\if
  l#1\relax \rlap{\raise.5\dimen7\hbox{\kern.4ex\char'47}}#1\relax \else\if
  L#1\relax\rlap{\raise.5\dimen7\hbox{\kern.7ex
  \char'47}}#1\relax\else\message{accent \string\soft \space #1 not
  defined!}#1\relax\fi\fi\fi\fi\fi\fi}

\end{document}